\documentclass[a4paper,11pt,english,french]{amsart}

\usepackage[utf8]{inputenc}
\usepackage{hyperref}

\usepackage[T1]{fontenc}
\usepackage{libertine}
\usepackage[libertine]{newtxmath}

\usepackage{amssymb,calrsfs,babel}
\renewcommand{\Bbb}{\mathbb}

\newcommand{\bL}{\mathbf{L}}
\newcommand{\bR}{\mathbf{R}}
\newcommand{\cC}{\mathcal{C}}

\newcommand{\cF}{\mathcal{F}}
\newcommand{\cG}{\mathcal{G}}

\newcommand{\cO}{\mathcal{O}}
\newcommand{\cP}{\mathcal{P}}
\newcommand{\cR}{\mathcal{R}}

\newcommand{\rR}{\overset{\circ}R}

\renewcommand{\leq}{\leqslant}
\renewcommand{\geq}{\geqslant}

\newcommand{\euc}{\mathrm{e}}

\DeclareMathOperator{\rk}{rk}
\DeclareMathOperator{\sign}{sign}
\DeclareMathOperator{\tr}{tr}

\DeclareMathOperator{\Vol}{Vol}

\DeclareMathOperator{\Ric}{Ric}

\DeclareMathOperator{\Sym}{Sym}

\newtheorem{theo}{Théorème}
\newtheorem{prop}[theo]{Proposition}
\newtheorem{lemm}[theo]{Lemme}

\theoremstyle{definition}

\theoremstyle{remark}

\newtheorem*{rema*}{Remarque}

\begin{document}

\author{Olivier Biquard}
\title[Non dégénérescence et singularités des métriques d'Einstein]{Non dégénérescence et singularités des métriques d'Einstein asymptotiquement hyperboliques en dimension 4}
\address{UPMC Université Paris 6 et École Normale Supérieure, UMR 8553 du CNRS}

\selectlanguage{english}
\begin{abstract}
  We prove that desingularizations of non degenerate Poincaré-Einstein metrics with $A_1$ singularities remain non degenerate. In principle this enables a recursive procedure to desingularize the other Fuchsian singularities. We illustrate this procedure by the $A_2$ case.
\end{abstract}

\maketitle
\selectlanguage{french}

\section*{Introduction}

Soit $(M_0^4,g_0)$ une variété d'Einstein asymptotiquement hyperbolique \cite{FefGra85} : cela signifie que $M_0$ possède un bord $\partial M_0=\{x=0\}$, où $x$ est une équation du bord, et que $g_0$ a un comportement «asymptotiquement hyperbolique» près du bord, à savoir $g_0\sim \frac{dx^2+\gamma_0}{x^2}$, où $\gamma_0$ est une métrique sur $\partial M_0$ ; en réalité, n'est bien définie par $g_0$ que la classe conforme $[\gamma_0]$, appelée infini conforme de $g_0$.

Supposons que $M_0$ ait un point singulier orbifold $p_0$ avec groupe $\Bbb{Z}_2$ (singularité $A_1$). Dans \cite{Biq13} on a montré que si $g_0$ est non dégénérée (au sens où la linéarisation de l'équation d'Einstein a un noyau $L^2$ trivial), et si la partie autoduale $\bR_+\in \Sym^2(\Omega_+)$ du tenseur de courbure vu comme endomorphisme symétrique des 2-formes (qui se décomposent en $\Omega^2=\Omega_+ \oplus \Omega_-$)  satisfait
\begin{equation}
  \label{eq:1}
  \det \bR_+^{g_0}(p_0)=0,
\end{equation}
alors il existe une famille $(g_t)$ de métriques d'Einstein lisses, asymptotiquement hyperboliques, sur une désingularisation topologique $M$ de $M_0$, dont la limite de Gromov-Hausdorff quand $t\rightarrow0$ est la métrique orbifold $g_0$. Ces métriques sont obtenues par recollement de $g_0$ avec des métriques de Eguchi-Hanson. En outre, si $\bR_+^{g_0}(p_0)$ est de rang maximal pour (\ref{eq:1}), c'est-à-dire de rang 2, alors l'espace des infinis conformes $\gamma$ de métriques d'Einstein orbifolds $g_0(\gamma)$ satisfaisant (\ref{eq:1}) est près de $\gamma_0$ une hypersurface $\cC_0 \subset \cC$, où $\cC$ est l'espace des métriques conformes sur $\partial M_0$. Les infinis conformes des métriques désingularisées $(g_t)$ sont nécessairement d'un côté de $\cC_0$ dans $\cC$ déterminé dans \cite{Biq16}, à savoir
\begin{equation}
  \label{eq:2}
  \det \bR_+^{g_0(\gamma)}(p_0)>0.
\end{equation}

Le premier objectif de cet article est de répondre à une question laissée en suspens dans \cite{Biq16} :
\begin{theo}\label{theo:1}
  Soit $(M_0^4,g_0)$ une variété d'Einstein asymptotiquement hyperbolique, avec une singularité orbifold $A_1$ au point $p_0$. Si $g_0$ est non dégénérée et le rang de $\bR_+^{g_0}(p_0)$ est égal à 2, alors les désingularisations $(g_t)$ sont non dégénérées pour $t>0$ petit.
\end{theo}
Une première application de ce théorème concerne la théorie (conjecturale) du degré qui doit compter le nombre de métriques d'Einstein sur $M$ d'infini conforme donné. Il permet de préciser le signe du changement de degré en traversant le mur, voir (\ref{eq:17}).

Une seconde application s'obtient grâce à une version plus générale du théorème \ref{theo:1} : l'énoncé s'applique aussi à la désingularisation partielle d'une singularité fuchsienne plus générale (singularités $A_k$, $D_k$ et $E_k$) par un espace ALE de rang 1 comme défini dans \cite{Biq16} ; par exemple, la désingularisation partielle d'une singularité $A_k$ laissant subsister une singularité $A_{k-1}$. En montrant que la désingularisation partielle demeure non dégénérée, le théorème permet l'itération du procédé de désingularisation. On peut donc en principe désingulariser n'importe quelle singularité fuchsienne de cette manière, pourvu qu'un certain nombre d'obstructions s'annulent. Déterminer précisément ces obstructions est ardu, et l'auteur n'a pu mener les calculs complets que pour la singularité suivante $A_2$ (le groupe cyclique $\Bbb{Z}_3$) :
\begin{theo}\label{th:2}
  Soit $(M_0,g_0)$ une variété d'Einstein, asymptotiquement hyperbolique, avec une singularité orbifold au point $p_0$ de type $A_2$, telle que le rang de $\bR_+^{g_0}(p_0)$ soit égal à 2. Alors, si une obstruction explicite sur le 2-jet de $\bR_+$ en $p_0$ s'annule, on peut désingulariser $(M_0,g_0)$ en une famille $(M,g_t)$ de métriques d'Einstein asymptotiquement hyperboliques sur une désingularisation topologique de $M_0$.
\end{theo}
En général, pour les singularités $A_k$, $D_k$ et $E_k$, on s'attend à trouver $k$ obstructions, la première étant (\ref{eq:1}), la $i$-ème portant sur le $2(i-1)$-jet de $\bR_+$ en $p_0$. En principe, la méthode itérative utilisée pour la singularité $A_2$ devrait s'étendre, mais obtenir exactement le bon nombre d'obstructions requiert de montrer certaines annulations qui ne sont pas a priori évidentes.

Une faiblesse du théorème \ref{th:2} est l'absence d'exemples concrets auxquels l'appliquer. On peut néanmoins imaginer une construction du type suivant : on sait désingulariser ces singularités dans le cadre Kähler-Einstein, pour des métriques asymptotiquement hyperboliques complexes. Or les infinis conformes des métriques asymptotiquement hyperboliques complexes (des structures CR) sont naturellement limites de métriques conformes (point de vue utilisé par exemple dans \cite{BiqHerRum07}). On devrait ainsi pouvoir approximer les métriques Kähler-Einstein asymptotiquement hyperboliques complexes par des métriques d'Einstein asymptotiquement hyperboliques réelles, et la désingularisation pour les métriques de Kähler-Einstein donnerait alors des exemples dans le cas réel.

Dans la section \ref{sec:calculs-sur-l-espace-ALE}, nous faisons des calculs précis sur la linéarisation de l'opérateur d'Einstein pour le développement formel qui intervient dans la procédure de désingularisation, ce qui nous permet dans la section \ref{sec:non-degenerescence} de démontrer le théorème \ref{theo:1}. Nous traitons alors, section \ref{sec:desing-A2}, le cas d'une singularité $A_2$, rendu possible par le processus itératif commencé dans \cite{Biq16} et la non dégénérescence prouvée dans le théorème \ref{theo:1}. Nous reportons à la section \ref{sec:germe-au-point} un autre ingrédient technique important, à savoir la génération de germes de métriques aux points singuliers à partir de l'infini conforme ; la méthode, nouvelle par rapport à \cite{Biq13}, permet d'énoncer les résultats en toute généralité. Finalement, la section \ref{sec:autres-singularites} considère une autre singularité (le groupe $\Bbb{Z}_4$, mais non inclus dans $SU(2)$) et donne un énoncé dans le cas de plusieurs points singuliers (proposition \ref{prop:plusieurs-points}).

\section{Calculs sur l'espace ALE}
\label{sec:calculs-sur-l-espace-ALE}

Rappelons le cadre utilisé dans \cite{Biq16}. On part d'une variété d'Einstein $(M_0^4,g_0)$, asymptotiquement hyperbolique, avec un point singulier orbifold $p_0$ de type $\Bbb{R}^4/\Gamma$, où $\Gamma$ est un sous-groupe fini de $SU_2$. On a $\Ric(g_0)=\Lambda g_0$ avec $\Lambda=-3$ (pour une meilleure clarté des formules, on écrira $\Lambda$ plutôt que sa valeur). On étudie le recollement de $(M_0,g_0)$ avec un espace asymptotiquement localement euclidien (ALE) $(Y^4,h_0)$ de Kronheimer \cite{Kro89a} : ici ALE signifie que $(Y,h_0)$ est asymptotique à $\Bbb{R}^4/\Gamma$ muni de sa métrique euclidienne $\euc$, où $\Gamma$ est bien le même groupe que celui au point $p_0\in M_0$. On suppose en outre que $Y$ est un orbifold de rang 1, c'est-à-dire $b_2^{orb}(Y)=1$. En particulier on utilisera le cas d'une singularité $A_k$ où $Y$ peut être choisi de rang 1 avec une unique singularité de type $A_{k-1}$ (si $k=1$, l'espace $Y$ se réduit à la métrique de Eguchi-Hanson sur $T^*\Bbb{C}P^1$).

On utilisera une fonction $R$ sur $Y$, déterminée par :
\begin{itemize}
\item près de l'infini, on a des coordonnées ALE $(x^i)$ telles que, après relèvement de l'action de $\Gamma$, la métrique $h_0$ se compare à la métrique euclidienne à l'ordre 4 : on prend $R$ le rayon dans ces coordonnées, alors $h_0-\euc=O(R^{-4})$ et plus généralement $\nabla^k(h_0-\euc)=O(R^{-4-k})$ ;
\item la fonction $R$ est prolongée en une fonction $C^\infty$ à l'intérieur de $Y$, telle que $R\geq 1$.
\end{itemize}

La désingularisation partielle de $g_0$ de \cite{Biq13,Biq16} est obtenue par un recollement entre $g_0$ et $t h_t$, où $h_t$ est un développement
\begin{equation}
 h_t^{[n]} = h_0 + t h_1 + \cdots + t^n h_n.\label{eq:3}
\end{equation}
En pratique on utilisera les ordres $n=1$ ou $2$ : le terme $h_1$ est quadratique à l'infini, d'asymptotique donnée par les termes d'ordre 2 de $g_0$ en $p_0$ ; le terme $h_2$ d'ordre 4 à l'infini, d'asymptotique donnée par les termes d'ordre 4 de $g_0$ en $p_0$ ; en outre, le développement $h_t$ est le début d'un développement formel pour une solution de l'équation $\Ric(t h_t)=\Lambda t h_t$, donc
\begin{equation}
 \Ric(h_t^{[n]}) = t \Lambda h_t^{[n]} + O(t^{n+1}).\label{eq:4}
\end{equation}

Notons $L_{h_t}$ la linéarisation de l'équation $\Ric-t\Lambda$ en $h_t$, c'est-à-dire
$$ L_{h_t} = \frac12 \nabla^*_{h_t}\nabla_{h_t} + \rR_{h_t} - \delta_{h_t}^* B_{h_t}, $$
où $B=\delta+\frac12 d\tr$ est l'opérateur de Bianchi, et $\delta^*$  la symétrisation de la dérivée covariante. A priori l'expression de $L_{h_t}$ n'a pas de sens car $h_t$ n'est pas une métrique ($h_1$ diverge à l'infini), mais il y a un développement formel, que nous utiliserons seulement à l'ordre 1 :
\begin{equation}
 L_{h_t}^{[1]} = L_0 + t L_1.\label{eq:5}
\end{equation}

On va maintenant tirer profit du calcul des termes d'ordre 2 du tenseur de Ricci fait dans \cite{Biq16} pour expliciter $L_1$ sur le noyau de $L_0$, ce qui permettra de comprendre le noyau de $L_{h_t}^{[1]}$. Bien sûr, pour avoir un noyau de dimension finie, on rajoute la condition de jauge $B_{h_0}h=0$, qui réduit à $L_0$ à l'opérateur $P_0 = \frac12 \nabla_{h_0}^*\nabla_{h_0} + \rR_{h_0}$. Celui-ci préserve la décomposition $\Sym^2 TY = \Bbb{R} \oplus \Sym_0^2 TY$. Rappelons que, $Y$ étant hyperkählérienne, le fibré des 2-formes autoduales $\Omega_+$ est trivialisé par les trois formes de Kähler $\omega_1$, $\omega_2$ et $\omega_3$, et donc $\Sym_0^2TY \simeq \Omega_+ \otimes \Omega_- \simeq \Bbb{R}^3 \otimes \Omega_-$. Dans cette identification, l'opérateur $P_0$ s'identifie à $d_-d_-^*$, donc son noyau $L^2$ est donné par la cohomologie $L^2$ de $Y$ : $\ker_{L^2}P_0 = \Bbb{R}^3 \otimes H^2_{L^2}(Y)$. Si $Y$ est de rang 1, alors $H^2_{L^2}(Y)$ est engendré par une seule forme antiautoduale $\Omega$, et donc $\ker_{L^2}P_0$ est engendré par les $o_i=\omega_i\circ \Omega$ (l'opération bilinéaire ici est la composition des 2-formes vues comme endomorphismes antisymétriques). Dans \cite{Biq16}, la structure complexe $J_1$ correspondant à $\omega_1$ est choisie comme celle de la résolution partielle de $\Bbb{C}^2/\Gamma$ ; alors $Y$ contient une unique courbe holomorphe $\Sigma$, et $\Omega$ est choisie Poincaré duale à $2\pi\Sigma$. Pour fixer complètement $h_0$, on fixe le volume de $\Sigma$ :
\begin{equation}
  \label{eq:6}
  \Vol \Sigma = 2\pi .
\end{equation}

Il reste à rappeler la construction de $h_1$ : l'équation (\ref{eq:4}) s'explicite en écrivant un développement formel de $\Ric-t\Lambda$, dont le premier terme est la linéarisation de $\Ric$ :
\begin{equation}
  \label{eq:7}
  \Ric(h_t)-t\Lambda h_t = t (L_0 h_1 - \Lambda h_0) + t^2 (Qh_1 - \Lambda h_1) + \cdots
\end{equation}
où $Q$ est quadratique en $h_1$. Notons $H_1=H_{ijkl} x^ix^jdx^kdx^l$ les termes d'ordre $2$ de $g_0$ en $p_0$, donc $g_0=\euc + H_1 + O(x^4)$. Quitte à faire agir un difféomorphisme local en $p_0$, on peut supposer que $H_1$ est en jauge de Bianchi : $B_\euc H_1=0$. Ces termes d'ordre 2 déterminent la courbure riemannienne en $p_0$, donc on peut écrire $R(H_1)$ pour la courbure de $g_0$ au point $p_0$, et nous considérerons particulièrement la partie $\bR_+(H_1)\in \Sym^2(\Omega_+)$ de l'opérateur de courbure (nous notons de manière différente la courbure $R$ du fibré $\Omega^2$ et l'opérateur de courbure $\bR$, qui diffèrent par le signe). Alors $h_1$ est solution du système
\begin{equation}\label{eq:8}
  \begin{split}
  L_0 h_1 & = \Lambda h_0 + \sum_1^3 \lambda_i o_i, \\ 
  B_{h_0}h_1 & = 0, \\
  h_1 & \sim H_1 \text{ à l'infini}, \\
  \int_\Sigma \phi_i & = 0.
\end{split}
\end{equation}
Les $\phi_i\in \Omega_-$ sont déterminées en écrivant $h_1=\lambda h_0+ \sum_1^3 \omega_i\circ \phi_i$, et la condition sur $\phi_i$ vise à éliminer l'ambiguïté sur la solution provenant du noyau $\langle o_i\rangle$. Les $\lambda_i$ sont complètement déterminés par $H_1$ : il y a une constante $\lambda'\neq0$ telle que
\begin{equation}
  \label{eq:9}
  \lambda_i = \frac {\lambda'}2 \langle \bR_+(H_1)\omega_1,\omega_i\rangle .
\end{equation}
L'annulation des trois coefficients $\lambda_i$ signifie $\bR_+(H_1)\omega_1=0$, donc que $\bR_+(H_1)$ a un noyau, ce qui est la source de la condition $\det \bR_+^{g_0}(p_0)=0$, puisqu'on peut toujours supposer, quitte à faire agit un élément de $SO_3$, que ce noyau est engendré par $\omega_1$.

La proposition suivante donne la raison profonde pour laquelle les métriques désingularisées $(g_t)$ seront non dégénérées : le terme de premier ordre $L_1$ est inversible sur le noyau de $L_0$ :
\begin{prop}\label{prop:L1}
  Supposons que $g_0$ satisfasse la condition $\det \bR_+^{g_0}(p_0)=0$. Alors $L_1$ préserve le sous-espace de dimension 3 engendré par les $o_i$, et y agit par la matrice $\bR_+^{g_0}(p_0)-\Lambda$. En particulier, puisque $\Lambda=\tr \bR_+^{g_0}$, si $\bR_+^{g_0}(p_0)$ est de rang 2, alors $L_1$ est inversible sur $\langle o_i\rangle$.
\end{prop}
Plus explicitement, si les valeurs propres de $\bR_+^{g_0}(p_0)$ sont $0$, $\Lambda_2$ et $\Lambda_3$, alors $\Lambda=\Lambda_2+\Lambda_3$ donc les valeurs propres de $L_1$ sur l'espace des $o_i$ sont $-\Lambda$, $-\Lambda_3$ et $-\Lambda_2$, toutes non nulles pourvu que $\Lambda_2$ et $\Lambda_3$ soient non nulles.
\begin{proof}
   De l'équation (\ref{eq:7}) on déduit, en notant $Q(h)=B(h,h)$ avec $B$ symétrique, que le terme $L_1$ dans (\ref{eq:5}) est
\begin{equation}
  \label{eq:10}
  L_1 = 2 B(h_1,\cdot) - \Lambda .
\end{equation}
Or les termes d'ordre 2 du tenseur de Ricci sont calculés dans \cite[lemme 3]{Biq16} : si on écrit le premier ordre de déformation $h_1=\lambda h_0+\sum_1^3\omega_i\circ \phi_i$, où $\phi_i\in \Omega_-$, et qu'on suppose $h_1$ en jauge de Bianchi ($B_{h_0}h_1=0$), alors la connexion induite sur $\Omega_+$ est modifiée à l'ordre 1 par la 1-forme $a=\sum_1^3 \omega_i\otimes *d\phi_i$, la courbure sur $\Omega_+$ par $R_+=\sum_1^3 \omega_i\otimes d*d\phi_i$, et les termes quadratiques de la partie sans trace du tenseur de Ricci, vue comme section de $\Omega_+\otimes \Omega_-$, sont
$$ Q(h_1) = \frac12 [a,a]_- - \phi(R_+), $$
où $\phi:\Omega_+\rightarrow \Omega_-$ est donné par $\phi(\omega_i)=\phi_i$.

Évaluons à présent $L_1$ sur $o_i=\omega_i\circ\Omega$. Rappelons de \cite[lemme 8]{Biq16} qu'au premier ordre dans la direction $h_1$, la courbure $R_+(h_1)$ est constante et égale à sa valeur à l'infini, $R_+(H_1)$. Par ailleurs, comme $d\Omega=0$, la variation au premier ordre de $a$, et donc de $R$, dans la direction $o_i$ est triviale, et dans (\ref{eq:10}) ne subsiste donc que
$$ L_1o_i = - o_i(R_+(H_1)) - \Lambda o_i . $$
Comme la courbure $R_+$ de $\Omega_+$ est l'opposé de l'opérateur de courbure $\bR_+$, la proposition s'en déduit.
\end{proof}

La proposition précédente ne prend pas en compte la jauge. En général, pour obtenir près d'une métrique $h$ les solutions de l'équation d'Einstein en jauge de Bianchi, on résoud l'équation $\Ric(g)-\Lambda g+\delta_g^*B_hg=0$, dont la linéarisation  en $g$ est
$$ P_h = L_h + \delta^*_h B_h = \frac12 \nabla^*_h\nabla_h + \rR_h. $$
Les éléments $o_i\in \ker L_{h_0}$ sont bien en jauge de Bianchi (c'est-à-dire dans le noyau $B_{h_0}$), mais rien ne dit que tel soit encore le cas à l'ordre 1 : comme pour $L$, on peut considérer les développements à l'ordre 1
\begin{equation}
  \label{eq:11}
  P_{h_t}^{[1]} = P_0 + t P_1, \quad B_{h_t}^{[1]} = B_0 + t B_1,
\end{equation}
et on a dans $P_{h_t}^{[1]}o_i$ un terme $t\delta^*_{h_0}B_1o_i$ qui est a priori du même ordre que $L_{h_t}^{[1]}o_i$. Pour y remédier on peut corriger la jauge à l'ordre 1 : soit $X_i$ un champ de vecteurs tel que
  \begin{equation}
    \label{eq:12}
    B_{h_0}\delta^*_{h_0}X_i = - B_1o_i .
  \end{equation}
  Comme $B\delta^*=\frac12 \nabla^*\nabla$ pour une métrique Ricci plate, et $B_1o_i = O(R^{-3})$ dans $Y$ (en effet $o_i=O(R^{-4})$ et $h_1=O(R^2)$ donne des termes $O(R)$ dans la connexion de Levi-Civita), une solution existe avec $X_i=O(R^{-1})$ et plus généralement $\nabla^kX_i=O(R^{-1-k})$. On considère alors la correction à l'ordre 1 de $o_i$ en
  \begin{equation}
    \label{eq:13}
    o_i^{[1]} = o_i + t \delta^*_{h_0}X_i .
  \end{equation}
  On obtient les contrôles
  \begin{equation}
    \label{eq:14}
    B_{h_t}^{[1]}o_i^{[1]} = O(t^2 R^{-1}), \quad L_{h_t}^{[1]}o_i^{[1]} = L_{h_t}^{[1]}o_i + O(t^2R^{-2}), 
  \end{equation}
d'où se déduit finalement, avec la proposition \ref{prop:L1} :
\begin{prop}\label{prop:P1}
  Notons $\bL$ la matrice symétrique $\bR_+^{g_0}(p_0)-\Lambda$. On a les contrôles
  \begin{align*}
    L_{h_t}^{[1]}o_i^{[1]} &= t \bL_i^j o_j^{[1]} + O(t^2R^{-2}) , \\
    P_{h_t}^{[1]}o_i^{[1]} &= t \bL_i^j o_j^{[1]} + O(t^2R^{-2}).
  \end{align*}
  Plus généralement, les dérivées $k$-ièmes sont contrôlées en $t^2R^{-2-k}$. \qed
\end{prop}

\section{Non dégénérescence}
\label{sec:non-degenerescence}

Nous montrons dans cette section que les solutions à l'équation d'Einstein construites par désingularisation de $g_0$ sont non dégénérées. Il nous faut rappeler le procédé de recollement de \cite{Biq13} : choisissons en $p_0$ des coordonnées $(x^i)$ telles qu'on ait un développement $g_0=\euc + H_1 + H_2 + \cdots$, avec $H_i$ d'ordre $2i$ et $B_\euc H_i=0$. Soit $r$ le rayon dans ces coordonnées, étendu sur $M_0$ de sorte qu'en dehors d'un voisinage de $p_0$ la fonction $r$ soit constante, égale à $1$. On fabrique une solution approchée $g_t$ sur une désingularisation topologique $M$ obtenue en recollant :
\begin{itemize}
\item la métrique $g_0$ sur $M_0$ sur la région $M^t = \{ r\geq \frac12 t^{\frac14} \}$ ;
\item la métrique $th_t^{[1]}$ sur $Y$ sur la région $Y^t = \{ R\leq 2t^{-\frac14} \}$.
\end{itemize}
L'identification entre les anneaux $A_M^t=\{ \frac12 t^{\frac14} \leq r \leq 2 t^{\frac14} \}\subset M_0$ et $A_Y^t=\{ \frac12 t^{-\frac14} \leq R \leq 2 t^{-\frac14}\}\subset Y$ se fait par une homothétie de rapport $\sqrt t$, en posant $r=\sqrt t R$, qui envoie la métrique $h_t^{[1]}$ sur $th_t^{[1]}$ (à noter que $h_t^{[1]}$ est une vraie métrique sur $Y^t$, elle n'est pas seulement formelle). Compte tenu que $h_t^{[1]}=h_0+th_1$ est construite de sorte que $h_1$ coïncide avec les termes d'ordre 2 de $g_0$ en $p_0$, on fait par le recollement une erreur $O(r^4)$ sur l'anneau $A_M^t$ (et l'erreur sur les dérivées d'ordre $k$ est en $r^{4-k}$). 

L'analyse sur $M$ est traitée dans des espaces de Hölder idoines, définis dans \cite[§~7]{Biq13}, et notés $C^{k,\alpha}_{\delta_0,\delta_\infty;t}$, où $\delta_0$ est le poids à la singularité et $\delta_\infty$ le poids à l'infini ; la norme à poids d'une section $s$ d'un fibré $E$ est définie par :
\begin{itemize}
\item près de l'infini conforme $\partial M=\{x=0\}$, (on prolonge $x$ à l'intérieur de $M$ par la valeur $1$)
  $$ \| x^{-\delta_\infty}s \|_{C^{k,\alpha}} ; $$
\item sur la région $A_M^t$, on utilise
$$ \sum_0^k \sup r^{\delta_0+k}|\nabla^ks| + |r^{\delta_0+k+\alpha}\nabla^ks|_\alpha  $$
où $|u|_\alpha$  est classiquement $|u|_\alpha=\sup \frac{|u(x)-u(y)|}{d(x,y)^\alpha}$ ;
\item sur $Y_M^t$ on utilise
$$ t^{\frac{\delta_0}2} \big( \sum_0^k \sup R^{\delta_0+k}|\nabla^ks| + |R^{\delta_0+k+\alpha}\nabla^ks|_\alpha \big) . $$
\end{itemize}
Toutes les normes sont prises par rapport à $g_t$. Le facteur $t^{\frac{\delta_0}2}$ permet la coïncidence des normes sur la région intermédiaire $A^t_M$ (homothétique à $A^t_Y$).

On étend aussi les tenseurs $o_i^{[1]}$ sur $Y$ comme dans \cite[§~13]{Biq13} : à l'infini sur $Y$, on a $o_i\sim \frac{\eta_i}{R^6}$ où $\eta_i$ est un 2-tenseur symétrique sur $\Bbb{R}^4$ dont les coefficients sont des formes quadratiques ; or on peut trouver sur $M_0$ des tenseurs $\bar o_i$ tels que $B_{g_0}\bar o_i=0$, $P_{g_0}\bar o_i=0$, et $L^2$ à l'infini ; le recollement de $o_i^{[1]}$ avec $t\bar o_i$ sur l'anneau $A^t$ fournit un 2-tenseur $o_{i,t}$ sur le recollement $M$.

\begin{lemm}\label{lemm:est-r}
  On a $P_{g_t}o_{i,t} = \bL_i^j o_{j,t} + r_{i,t}$, avec
  $$\begin{cases}
    |r_{i,t}|_{g_t} \leq c t r^{-4} & \text{ sur } M^t, \\
    |r_{i,t}|_{g_t} \leq c R^{-2} & \text{ sur } Y^t\setminus A_Y^t,
  \end{cases}$$
  et les estimations qui en découlent sur les dérivées ($|\nabla^kr_{i,t}|_{g_t} \leq c t r^{-4-k}$ sur $M^t$, etc.) En particulier, $\|r_{i,t}\|_{C^{2,\alpha}_{\delta_0+2,\delta_\infty;t}} \le c t^{\frac12+\frac{\delta_0}4}$.
\end{lemm}
Dans la région de transition, les deux estimations ne sont pas du même ordre ($R^{-2}=tr^{-2}$), ce qui est normal car les $o_{i,t}$ ne satisfont pas la même équation sur $Y$ et sur $M_0$ (sur $M_0$ on a $P_{g_t}o_{i,t}=0$), donc le recollement fait nécessairement apparaître un terme d'erreur de l'ordre de $o_{i,t}$, donc en $tr^{-4}$.
\begin{proof}
  On a $P_{\frac{g_t}t}=tP_{g_t}$ et les normes des 2-tenseurs pour $g_t$ et $\frac{g_t}t$ diffèrent d'un facteur $t$, donc l'estimation sur $Y^t$ résulte de la proposition \ref{prop:P1}. Sur $M^t\setminus A_M^t$, puisque $P_{g_0}\bar o_i=0$, l'erreur est $-L_i^j t\bar o_j$ qui est effectivement en $tr^{-4}$. Enfin, sur l'anneau de transition $A_M^t$, les termes principaux de $o_i^{[1]}$ et $t\bar o_i$ coïncident et sont tous deux en $tr^{-4}$, donc l'erreur dûe au recollement est en $tr^{-2}$, qui dans $P_{g_t}o_{i,t}$ donne une erreur en $tr^{-4}$ aussi. On en déduit l'estimation sur la norme à poids, où le plus mauvais terme est celui sur $M^t$.
\end{proof}

\begin{lemm}\label{lemm:inv-P}
  On suppose que $g_0$ est une métrique d'Einstein non dégénérée telle que $\bR_+^{g_0}(p_0)$ soit de rang 2. Alors l'opérateur $P_{g_t}:C^{2,\alpha}_{\delta_0,\delta_\infty;t}\rightarrow C^\alpha_{\delta_0+2,\delta_\infty;t}$ est inversible pour $t>0$ petit, et la norme de son inverse explose en $t^{-1}$ quand $t\rightarrow0$.
\end{lemm}
\begin{proof}
Notons $\cO_t = \langle o_{i,t} \rangle$, que nous munissons de la norme $|\sum_1^3x_io_{i,t}|^2=\sum_1^3|x_i|^2$. Nous pouvons considérer $\bL$ comme un endomorphisme symétrique de $\cO_t$, et notons $r(\sum_1^3x_io_{i,t})=\sum_1^3x_ir_{i,t}$. L'opérateur $P_{g_t}$ est étudié dans \cite[§~9]{Biq13} : un supplémentaire $S$ de $\cO_t$ est défini par la condition
$$ \int_\Sigma \phi_i = 0, \quad\text{ où } h = \lambda h_0 + \sum_1^3 \omega_i\circ\phi_i \text{ sur }Y^t. $$
Le problème
$$ P_{g_t}h + x = u \text{ avec } h\in S, x\in \cO_t,$$
est résolu avec un contrôle indépendant de $t$ :
\begin{equation}
 t^{-\frac{\delta_0}2} \|h\|_{C^{2,\alpha}_{\delta_0,\delta_\infty;t}} + |x| \leq C t^{-\frac{\delta_0}2} \|u\|_{C^\alpha_{\delta_0+2,\delta_\infty;t}}.\label{eq:15}
\end{equation}
(Stricto sensu ce n'est pas exactement $\cO_t$ qui est considéré dans \cite{Biq13} mais il n'y a aucune différence dans l'estimation).

Écrivons à présent un tenseur arbitraire $h=s+x$ avec $s\in S$ et $x\in \cO$. En appliquant cette estimation et le lemme \ref{lemm:est-r} :
\begin{align*}
  t^{-\frac{\delta_0}2} \|s\|_{C^{2,\alpha}_{\delta_0,\delta_\infty;t}} + |\bL x|
  & \leq C t^{-\frac{\delta_0}2} \|P_{g_t}h + r(x)\|_{C^\alpha_{\delta_0+2,\delta_\infty;t}} \\
  & \leq C \big( t^{-\frac{\delta_0}2} \|P_{g_t}h\|_{C^\alpha_{\delta_0+2,\delta_\infty;t}} + t^{\frac12-\frac{\delta_0}4} |x| \big).
\end{align*}
Si $\bL$ est inversible, on déduit, pour $t$ assez petit,
\begin{equation}
 t^{-\frac{\delta_0}2} \|s\|_{C^{2,\alpha}_{\delta_0,\delta_\infty;t}} + |x| \leq C t^{-\frac{\delta_0}2} \|P_{g_t}h\|_{C^\alpha_{\delta_0+2,\delta_\infty;t}}\label{eq:16}
\end{equation}
ce qui prouve l'injectivité de $P_{g_t}$, avec une estimation uniforme. On remarquera que $\|x\|_{C^{2,\alpha}_{\delta_0,\delta_\infty;t}} \sim t^{\frac{\delta_0}2-1}|x|$, d'où l'assertion sur la norme de l'inverse.
\end{proof}

\begin{proof}[Démonstration du théorème \ref{theo:1}]
  Les métriques d'Einstein sont obtenues comme perturbations du recollement $g_t$ considéré dans le lemme précédent. Pour montrer le théorème, nous avons besoin du raffinement du recollement fait dans \cite[§~14]{Biq13} : au lieu de recoller seulement $th_t^{[1]}$ à $g_0$, on prend un terme de plus dans le développement formel (\ref{eq:3}), en considérant $h_t^{[2]}=h_t^{[1]}+t^2h_2$, où $h_2$ est d'ordre 4 à l'infini, $h_2=O(R^4)$, et ses termes d'ordre 4 à l'infini coïncident avec ceux de $g_0$ en $p_0$. On obtient alors une meilleure approximation, que nous noterons $g_t^{[2]}$, de la métrique d'Einstein. Enfin, il faut aussi noter la dépendance de toute la construction par rapport à l'infini conforme $\gamma_0$ de $g_0$, et si celui-ci varie nous noterons explicitement cette dépendance par $g_0(\gamma)$, $g_t(\gamma)$, etc.

  La première observation est que tout ce que nous avons fait avant demeure inchangé si on remplace $g_t$ par $g_t^{[2]}$. Le point essentiel est la proposition \ref{prop:P1} : par rapport à $h_t^{[1]}$, la métrique $h_t^{[2]}$ comporte un terme additionnel qui est $O(t^2R^4)$, et $o_i^{[1]}$ est $O(R^{-4}+tR^{-2})$, donc on obtient dans $P_{h_t}^{[1]}o_i^{[1]}$ un terme $O(t^2R^{-2}+t^3)$, c'est-à-dire $O(t^2R^{-2})$ sur la région $Y^t$. Le terme d'erreur dans la proposition \ref{prop:P1} est ainsi inchangé, ainsi que les estimations sur $P_{g_t^{[1]}}$.

  La seconde observation est que la désingularisation d'Einstein s'écrit (cf. l'équation (111) dans \cite{Biq13})
  $$ \hat g_t = g_t^{[2]}(\gamma_t) + u_t, \quad \|u_t\|_{C^{2,\alpha}_{\delta_0,\delta_\infty;t}} = O(t^{\frac32+\frac{\delta_0}4}), $$ 
  où $\gamma_t$ est un chemin d'infinis conformes : $\gamma_t=\gamma_0+t\gamma_1+O(t^{\frac32})$. Or, d'une part, la dépendance de $\gamma_t$ par rapport à $t$ entraîne sur $Y$ la modification à l'ordre $t$ de l'asymptotique du terme $h_1$, et donc perturbe $g_t^{[2]}$par un terme $O(t^2R^2)$, qui à nouveau introduit un terme d'erreur qui ne modifie pas les estimations ; d'autre part, on a sur $Y^t$ les estimations
$$ \big| \frac{u_t}t \big|_{\frac{g_t}t} \leq t^{-\frac{\delta_0}2} R^{-\delta_0} \|u_t\|_{C^{2,\alpha}_{\delta_0,\delta_\infty;t}} $$
(et les estimations similaires pour les dérivées), ce qui provoque dans les contrôles de la proposition \ref{prop:P1} une erreur dans $P_{\frac{g_t}t}o_{i,t}$ de l'ordre de
$$ t^{-\frac{\delta_0}2} R^{-6-\delta_0} \|u_t\|_{C^{2,\alpha}_{\delta_0,\delta_\infty;t}} $$
par rapport à la métrique $\frac{g_t}t$ (toujours sur $Y_t$), et donc, revenant à la métrique $g_t$, un terme d'erreur $\epsilon_{i,t}$ dans $P_{g_t}o_{i,t}$ controlé en
$$ \|\epsilon_{i,t}\|_{C^{2,\alpha}_{\delta_0,\delta_\infty;t}} \leq t^{-1} \|u_t\|_{C^{2,\alpha}_{\delta_0,\delta_\infty;t}} = O(t^{\frac12+\frac{\delta_0}4}) $$
qui est du même ordre que l'erreur sur $r_{i,t}$ dans le lemme \ref{lemm:est-r}, et l'erreur sur la partie $M^t$ contribue encore moins.

  Finalement, il en résulte que les estimations des erreurs restent les mêmes, donc le raisonnement fait pour le lemme \ref{lemm:inv-P} demeure inchangé.
\end{proof}

La démonstration du théorème donne plus que la non dégénérescence des métriques d'Einstein désingularisées, elle donne aussi le signe de la modification du degré d'Anderson \cite{And08} du problème de Dirichlet à l'infini pour les métriques d'Einstein quand on passe le «mur» $\cC_0$. Comme ce degré n'est toujours pas rigoureusement défini dans notre situation, nous nous contentons de considérations conjecturales et ne justifions pas toutes les assertions qui suivent.

Le degré d'Anderson est défini en comptant le nombre de métriques d'Einstein sur $M$, d'infini conforme donné, avec un signe donné par le nombre de valeurs propres strictement négatives de la linéarisation $P$. Dans \cite{Biq16} on montrait que, partant d'une métrique d'Einstein orbifold $g_0$ telle que $\rk \bR_+^{g_0}(p_0)=2$, les désingularisations sont du côté du mur donné par l'inégalité (\ref{eq:2}), mais le signe de cette solution n'était pas calculé. Nous pouvons y remédier grâce au calcul plus précis que nous avons fait.

Quand $t\rightarrow0$, la linéarisation $P$ de l'équation d'Einstein tend d'une part vers la linéarisation $P_{g_0}$ sur l'orbifold $M_0$, d'autre part sur l'instanton gravitationnel $Y$ se comporte comme $t^{-1}P_{h_0}$, qui est un opérateur positif ou nul, avec noyau de dimension 3 engendré par les $o_i$. Le spectre de la linéarisation $P$ se découple donc, quand $t\rightarrow0$, en 3 parties :
\begin{itemize}
\item le spectre de $P_{g_0}$ ;
\item la partie strictement positive du spectre de $P_{h_0}$, qui tend vers $+\infty$ à la vitesse $t^{-1}$ ;
\item une partie de dimension 3, correspondant au noyau de $P_{h_0}$, dont les valeurs propres tendent vers celles de $\bR_+^{g_0}(p_0)-\Lambda$, à savoir $-\Lambda_2$, $-\Lambda_3$ et $-\Lambda=+3$ si $\Lambda_2$ et $\Lambda_3$ sont les valeurs propres non nulles de $\bR_+^{g_0}(p_0)$.
\end{itemize}
On voit donc que par rapport à $P_{g_0}$, la linéarisation de la désingularisation compte un nombre de valeurs propres strictement négatives augmenté du nombre de valeurs propres strictement négatives de $\bR_+^{g_0}(p_0)-\Lambda$, donc le signe de la désingularisation est égal au signe de $g_0$ multiplié par $\sign(\Lambda_2\Lambda_3)$. Autrement dit, le changement de degré quand on passe du domaine $\det \bR_+^{g_0(\gamma)}(p_0)<0$ au domaine $\det \bR_+^{g_0(\gamma)}(p_0)>0$ est donné par
\begin{equation}
  \label{eq:17}
  \sign(g_0) \sign(\Lambda_2\Lambda_3).
\end{equation}

\section{La singularité $A_2$}
\label{sec:desing-A2}

Nous montrons à présent que la non dégénérescence démontrée dans le théorème \ref{theo:1} permet de mener à bien le programme de désingularisation esquissé dans \cite[§~6]{Biq16}. Nous commençons cette section par le cas d'une singularité $A_2$.

Rappelons d'abord brièvement la suite de la procédure de désingularisation. On utilise le développement formel $h_t^{[2]}$ satisfaisant (\ref{eq:4}) à l'ordre $n=2$ modulo les obstructions : le terme $h_2$ est obtenu en résolvant un système analogue à (\ref{eq:8}), et on obtient ainsi une solution de l'équation
\begin{equation}
  \label{eq:18}
  \Ric(h_t^{[2]}) = t\Lambda h_t^{[2]} + \sum_1^3 (t\lambda_i+t^2\mu_i)o_i + O(t^3).
\end{equation}
Les coefficients $\lambda_i$ sont écrits dans (\ref{eq:9}) et $\mu_1$ est déterminé dans \cite{Biq16}.

Pour faire le recollement, il faut aussi améliorer la coïncidence de $g_0$ avec $h_0$ : on sait que $h_0-\euc=O(R^4)$, en fait, plus précisément, il existe un développement à l'infini
\begin{equation}
  \label{eq:19}
  h_0 = \euc + \frac{K_2}{R^6} + \frac{K_3}{R^8} + \cdots
\end{equation}
avec $K_j$ un 2-tenseur symétrique dont les coefficients sont des polynômes homogènes de degré $2j$, donc $K_2=K_{2,ijkl}x^ix^jdx^kdx^l$, etc. Le terme $\frac{K_2}{R^6}$, d'ordre 4, est au niveau du recollement de la même taille que $t^2h_2$ et doit donc être pris en compte. Pour cela, on modifie dans \cite[§~14]{Biq13} la métrique $g_0$ par un terme $t^2k_2$ défini sur $M_0$, qui est $L^2$ près de $\partial M_0$, et satisfait
\begin{equation}
  \label{eq:20}
  P_{g_0}k_2 = 0, \quad B_{g_0}(k_2) = 0, \qquad k_2 \sim \frac{K_2}{r^6} \text{ près de }p_0.
\end{equation}
On peut aussi résoudre le système analogue sur $k_3$, les termes non linéaires de Ricci n'interviennent pas car ils sont d'ordre plus grand ($O(R^{-10})$ pour $\frac{K_2}{R^6}$).

Le recollement de $th_t^{[2]}$ avec $g_0+t^2k_2$ sur l'orbifold produit une solution approchée de l'équation d'Einstein, que l'analyse développée dans \cite{Biq13} permet de déformer en une solution $g_t$ de l'équation
\begin{equation}
  \label{eq:21}
  \Ric(g_t)-\Lambda g_t = \sum_1^3 \lambda_i(t) o_{i,t}, \quad \lambda_i(t)=t\lambda_i+t^2\mu_i+O(t^{\frac52}).
\end{equation}
Ici l'infini conforme de $g_t$ ne varie pas et les tenseurs $o_{i,t}$ sont ceux considérés dans la section \ref{sec:non-degenerescence}, mais convenablement projetés sur le noyau de l'opérateur de Bianchi $B_{g_t}$ de sorte que l'équation (\ref{eq:21}) soit possible ; pour éviter d'alourdir les notations, nous utilisons encore le même symbole.

Toute la construction dépend de deux paramètres importants :
\begin{itemize}
\item un paramètre $\varphi$ de recollement de l'instanton gravitationnel avec l'orbifold $M_0$ : on peut appliquer avant recollement un élément de $SO_4$ ; la valeur de $\varphi$ pour $t=0$ est implicitement fixée par le choix d'identification de $T_{p_0}M_0$ avec $\Bbb{R}^4$ de sorte que $R_+^{g_0}(\omega_1)=0$ ; dans la suite, il nous suffira de prendre $\varphi\in Sp_1$, dont l'algèbre de Lie est $\Omega_+\Bbb{R}^4$ ;
\item l'infini conforme $[\gamma]$ sur le bord à l'infini $\partial M$.
\end{itemize}
Au besoin nous noterons cette dépendance par $g_t(\gamma)$, $\lambda_i(\gamma)$, etc.

La famille de métriques d'Einstein désingularisées est obtenue en variant les paramètres par rapport à $t$ : un choix adéquat de $\gamma(t)$ et $\varphi(t)$ permet d'obtenir une métrique $g_t(\varphi(t),\gamma(t))$ d'Einstein en tuant les obstructions présentes dans (\ref{eq:21}). Rappelons que la condition (\ref{eq:1}), écrite dans une base convenable, dit que les obstructions s'annulent au premier ordre : $\lambda_i=0$ ; en revanche il n'y aucune raison pour l'annulation des $\mu_i$. De \cite[§~12]{Biq13} résulte les faits suivants :
\begin{itemize}
\item si $\rk \bR_+^{g_0}(p_0)=2$, alors les coefficients $\lambda_2(t)$ et $\lambda_3(t)$ peuvent être annulés par un choix adéquat du paramètre de recollement $\varphi(t)$ ;
\item il existe une métrique conforme infinitésimale $\dot \gamma_1$ telle que
\begin{equation}
  \label{eq:22}
  \frac{\partial \lambda_1}{\partial \gamma} (\dot \gamma_1) = 1 ;
\end{equation}
aussi l'obstruction $\lambda_1(t)$ peut être tuée grâce par un choix de $\gamma(t)=\gamma_0+f(t)\dot \gamma_1$ satisfaisant
\begin{equation}
  \label{eq:23}
  f(t) = - t \mu_1 + O(t^{\frac32}).
\end{equation}
\end{itemize}
Au total, on obtient la désingularisation d'Einstein $g_t(\varphi(t),\gamma(t))$ voulue sur $M$.

Partons de $(M_0,g_0)$, Einstein asymptotiquement hyperbolique, non dégénérée, avec un point singulier de type $A_k$ en $p_0$. Il existe donc une désingularisation partielle $g_t=g_t(\varphi(t),\gamma(t))$ obtenue en recollant un instanton gravitationnel orbifold de rang 1, $Y$, avec un point singulier de type $A_{k-1}$ en un point $p_1$ placé sur l'unique courbe holomorphe $\Sigma\subset Y$. En outre, il est montré dans \cite[lemme 18]{Biq16} qu'en écrivant, grâce au choix d'une base de diagonalisation de $\bR_+^{g_0}$ en $p_0$,
$$ \bR_+^{g_0}(p_0) = \begin{pmatrix} 0 & & \\ & \Lambda_2 &  \\ & & \Lambda_3 \end{pmatrix} $$
on a
\begin{equation}
 R_+^{g_t}(p_1) = t R_+^{g_0}(p_0) + t^2 B + O(t^3),\label{eq:24}
\end{equation}
où $B$ est une matrice dont le coefficient $B_{11}(p_1)$ au point $p_1$ est donné en termes du germe de $g_0$ en $p_0$ par la formule
\begin{equation}\label{eq:25}
 B_{11}(p_1) = A(p_0) := -(k-1)\Lambda_2\Lambda_3  + \frac{k+1}{16} \big\langle (\nabla^2_{11}+\nabla^2_{22}-\nabla^2_{33}-\nabla^2_{44})R(p_0)\omega_1,\omega_1 \big\rangle.
\end{equation}
Notons que les termes avec des dérivées secondes ont un sens car la diagonalisation de l'action du groupe $A_k$ décompose l'espace tangent en $p_0$ en une somme $\Bbb{C}\oplus\Bbb{C}$. En revanche on peut échanger ces deux facteurs (ce qui correspond à changer la manière de recoller $Y$ à $M_0$), ce qui change le signe du second terme dans (\ref{eq:25}). En particulier, on déduit
\begin{equation}
\det R_+^{g_t}(p_1)=A(p_0) \Lambda_2\Lambda_3t^4 + O(t^5).\label{eq:26}
\end{equation}
Puisque $g_t$ est non dégénérée par le théorème \ref{theo:1}, on peut désingulariser $g_t$ au point $p_1$ pourvu que $\det R_+^{g_t}(p_1)=0$. Un changement d'ordre $t$ de l'infini conforme provoque une modification d'ordre $t$ de $\bR_+^{g_0}(p_0)$ et donc une modification de $\det R_+^{g_t}(p_1)$ par un terme d'ordre $t^5$, donc il est clair que le coefficient $A(p_0)$ est une obstruction à poursuivre la désingularisation. C'est la seule, car on montrera dans la section \ref{sec:germe-au-point} :
\begin{prop}\label{prop:germe-A2}
  Il existe une métrique conforme infinitésimale $\dot \gamma_2$ sur $\partial M$, telle que la perturbation $\dot g_2$ de $g_0(\gamma)$ dans la direction $\dot \gamma_2$ satisfasse au point $p_0$ :
  \begin{itemize}
  \item les termes d'ordre 2 de $\dot g_2$ sonts nuls, autrement dit $\dot g_2=O(r^4)$, en particulier $\frac{\partial \lambda_i}{\partial \gamma}(\dot \gamma_2)=0$ ;
  \item la dérivée de $\langle (\nabla^2_{11}+\nabla^2_{22}-\nabla^2_{33}-\nabla^2_{44})R(p_0)\omega_1,\omega_1 \rangle$ dans la direction $\dot g_2$ est égale à $1$, donc en particulier $\frac{\partial A(p_0)}{\partial \gamma}(\dot \gamma_2)\neq 0$.
  \end{itemize}
\end{prop}
Il résulte de la proposition et des formules (\ref{eq:25}) (\ref{eq:26}) que si $A(p_0)=0$, alors il y a une solution $g_t(\varphi(t),\gamma(t))$ avec $\gamma(t)=\gamma_0+f_1(t)\dot \gamma_1+f_2(t)\dot \gamma_2$ des équations
\begin{equation}
  \label{eq:27}
  (\Ric-\Lambda)(g_t(\varphi(t),\gamma(t))) = 0 , \quad \det R_+^{g_t(\varphi(t),\gamma(t))}(p_1)=0.
\end{equation}
Autrement dit $g_t(\varphi(t),\gamma(t))$ est une métrique d'Einstein dont la singularité en $p_1$ est partiellement désingularisable. Dans le cas $A_2$, la singularité en $p_1$ est une singularité $A_1$, et la désingularisation est donc finie. On a ainsi montré la première partie du théorème \ref{th:2} :
\begin{theo}\label{th:A2}
  Soit $(M_0,g_0)$ une variété d'Einstein asymptotiquement hyperbolique, non dégénérée, avec une singularité de type $A_2$ au point $p_0$. Si $\rk \bR_+^{g_0}(p_0)=2$, ce qui implique (\ref{eq:1}), et si $A(p_0)=0$ où $A(p_0)$ est le coefficient défini par la formule (\ref{eq:25}) avec $k=2$, alors il existe une famille de désingularisations d'Einstein de $g_0$ sur une désingularisation topologique de $M_0$.
\end{theo}
Plus précisément, on obtient une famille à deux paramètres $(t_1,t_2)$, où $t_1$ est le paramètre pour la désingularisation partielle de $M_0$ avec une singularité résiduelle $A_1$, et $t_2$ le paramètre de la seconde désingularisation. Cela correspond au fait que $A_2$ est de rang 2.

L'application du théorème \ref{theo:1} implique que les désingularisations obtenues restent non dégénérées.

\section{Germe au point singulier et infini conforme}
\label{sec:germe-au-point}

Ici nous déterminons les germes de variation de $g_0$ en $p_0$ réalisables par une variation de l'infini conforme de $g_0$, en généralisant \cite[§~10--11]{Biq13} qui se limitait au cas des germes d'ordre 2. Une méthode plus intrinsèque est nécessaire pour traiter les germes d'ordre plus élevé ; incidemment nous comblons une légère lacune dans le traitement de \cite{Biq13}, relevée par Morteza et Viaclovsky \cite{MorVia}, mais sans incidence sur les résultats. La méthode ici est plus simple car elle ne fait appel qu'au théorème de continuation unique pour les opérateurs elliptiques de \cite{Maz91c} au lieu du théorème de continuation unique de \cite{Biq08} pour le problème d'Einstein, non elliptique à cause de l'action des difféomorphismes.

Pour traiter le cas de l'opérateur $P=\frac12\nabla^*\nabla-\rR$ agissant sur les sections de $\Sym_0^2(T^*M_0)$, on doit aussi considérer les opérateurs $B\delta^*$ agissant sur les sections de $TM$, et $\Delta-2\Lambda=\Delta+6$ agissant sur les fonctions. Le lien entre ces opérateurs provient des remarques suivantes :
\begin{itemize}
\item si $B\delta^*X=0$ alors $P\delta^*X=0$ et en particulier $P(\delta^*X)_0=0$, où $(\delta^*X)_0$ est la partie sans trace ;
\item de même, puisque
$$ B\delta^* = \frac12 (\nabla^*\nabla-\Lambda) = \frac12 (\Delta-2\Lambda), $$
où $\Delta=dd^*+d^*d$ est le laplacien de Hodge-De Rham sur les 1-formes (identifiées aux vecteurs par la métrique), une solution de $(\Delta-2\Lambda)f=0$ donne naissance à une solution $df$ de $B\delta^*df=0$, et donc à une solution $(\nabla df)_0$ de $P((\nabla df)_0)=0$.
\end{itemize}

Soit $L=L_{g_0}$ l'un quelconque des opérateurs ci-dessus, agissant sur les sections du fibré $E$ (ou, plus généralement, un laplacien géométrique de type $\nabla^*\nabla+\cR$, où $\cR$ est un terme de courbure). Notons $\cP_k$ l'espace des polynômes harmoniques homogènes de degré $k$ sur $\Bbb{R}^4$. Les termes principaux de $L$ en $p_0$ s'identifient à ceux pour la métrique euclidienne $\euc$, à savoir le laplacien scalaire sur chaque coordonnée de $s$. Le terme principal d'une solution de $Ls=0$ est donc à coefficients harmoniques pour $\euc$, donc donné, s'il est d'ordre $k$, par des éléments de $\cP_k$, autrement dit $s \sim \sigma$ avec $\sigma\in \cP_k\otimes E$ ; si au contraire $s$ diverge en $p_0$, son terme principal sera donné par un comportement « dual » de type $\frac \sigma{r^{2+2k}}$, toujours avec $\sigma \in \cP_k\otimes E$. Bien entendu, si le point est orbifold de groupe $\Gamma$, on doit se restreindre à l'espace $(\cP_k\otimes E)^\Gamma$ des germes d'ordre $k$ invariants sous $\Gamma$.

Au bord à l'infini $\partial M_0$, un tel laplacien géométrique a pour terme dominant $-(x\partial_x)^2+3x\partial_x+A$, où $A$ est un opérateur linéaire auto-adjoint, et l'écriture de l'opérateur se fait dans une trivialisation orthonormale du fibré à l'infini ; les valeurs propres $\lambda$ de $A$ permettent de décomposer à l'infini le fibré $E=\oplus E_\lambda$, et le comportement asymptotique est donné sur chaque composante par l'opérateur scalaire $-(x\partial_x)^2+3x\partial_x+\lambda$, dont les solutions sont $x^{\delta^\pm}$, où $\delta^\pm=\frac32\pm\sqrt{\frac94+\lambda}$ sont les poids critiques, voir \cite{Maz91} pour l'analyse de ces opérateurs.

La base de notre traitement est l'intégration par parties suivantes \cite[(92)]{Biq13} : supposons qu'on ait deux solutions $s_\pm$ et $Ls_\pm=0$, avec les comportements duaux suivants :
\begin{itemize}
\item au point $p_0$, on a $s_+ \sim \sigma_+$ et $s_- \sim \sigma_- r^{-2-2k}$, avec $\sigma_\pm \in \cP_k\otimes E$ ;
\item à l'infini, $s_\pm \sim x^{\delta^\mp} \tau_\pm$, où $\tau_\pm$ est une section de $E$ sur $\partial M_0$, et $\delta^\pm$ sont des poids critiques duaux ($\delta^++\delta^-=3$, $\delta^+>\delta^-$) de $L$ à l'infini.
\end{itemize}
Alors on a
\begin{equation}
  \label{eq:31}
  (\sigma_+,\sigma_-) = \frac{\delta^+-\delta^-}{2k+2} (\tau_+,\tau_-)
\end{equation}
où le premier produits scalaire est le produit scalaire standard de $\cP_k\otimes E$, et le second est le produit scalaire $L^2$ sur les sections de $E$ sur $\partial M_0$.

\subsection*{Le cas des fonctions} Nous n'avons pas directement besoin de ce cas, mais il permet d'expliquer les idées plus simplement, et nous nous référerons ensuite à la démonstration faite ici. Soit $\delta_0^\pm$ les deux poids critiques de $L=\Delta-2\Lambda=\Delta+6$ à l'infini. L'inversibilité de $L$ dans $L^2$ implique immédiatement qu'étant donné une fonction $\tau$ sur $\partial M_0$, il existe une unique fonction $s$, solution de l'équation
\begin{equation}
 Ls=0, \qquad s\sim x^{\delta_0^-} \tau \text{ à l'infini}.\label{eq:32}
\end{equation}
Le terme principal de $s$ en $p_0$ est alors un polynôme harmonique d'un certain degré $k$.
\begin{lemm}\label{lem:fonctions}
  Pour tout $k\geq 1$ et tout germe de fonction harmonique $\sigma\in \cP_k^\Gamma$ d'ordre $k$ en $p_0$, il existe une fonction $\tau$ sur $\partial M_0$ telle que la solution du système (\ref{eq:32}) ait pour terme principal $\sigma$ en $p_0$.
\end{lemm}
\begin{proof}
  Étant donné $\sigma\in \cP_k^\Gamma$, posons $s_0=\frac \sigma{r^{2k+2}}$ qui est en $r^{-k-2}$, alors $Ls_0=O(r^{-k-2})$, donc ce terme d'erreur peut être compensé par un terme $s_1$ en $r^{-k}$ ; de proche en proche, on peut formellement corriger $s_0$ en un $\bar s$ défini près de $p_0$ tel que $\bar s\sim \frac \sigma{r^{2k+2}}$ et $L\bar s$ soit $L^2$ près de $p_0$. Utilisant l'inversibilité de $L$ sur $M_0$, on déduit finalement l'existence de $s$ définie sur $M_0$, telle que
\begin{equation}
Ls=0, \qquad s \sim \frac \sigma{r^{2k+2}} \text{ en }p_0, \qquad s \sim x^{\delta_0^+} \tau \text{ à l'infini}.\label{eq:33}
\end{equation}
(Le fait que $s$ soit $L^2$ à l'infini impose que $s=O(x^{\delta_0^+})$, l'autre poids critique $\delta_0^-$ ne peut pas apparaître). On définit ainsi un opérateur $S(\sigma)=\tau$, donc
$$ S: \cP_k \longrightarrow C^\infty(\partial M_0). $$
A priori, $S$ est mal défini, car la solution $s$ dans (\ref{eq:33}) est définie à l'ambiguïté près des solutions obtenues à partir de germes dans des $\cP_\ell$ pour $\ell<k$, mais cette ambiguité peut être levée en décidant que, pour le produit scalaire $L^2$,
$$ S(\cP_k) \perp \oplus_{\ell<k} S(\cP_\ell). $$

Le point essentiel de la démonstration consiste à montrer que $S$ est injective. En effet, si ce n'était pas le cas, on disposerait d'une solution $s$ du système (\ref{eq:33}) avec $\tau=0$. Le théorème de continuation unique dans cette situation \cite{Maz91c} implique qu'en réalité $s=0$, ce qui est impossible si $\sigma\neq 0$.

L'énoncé du lemme se déduit alors par passage au dual, grâce à (\ref{eq:31}). Étant donné $\tau\in C^\infty(\partial M_0)$, il existe une unique solution $s$ du problème de Dirichlet à l'infini (\ref{eq:32}), à savoir $Ls=0$ et $s \sim \tau x^{\delta_0^-}$ à l'infini. Notons $E_k$ l'espace des $\tau\in C^\infty(\partial M_0)$ tels que $s=O(r^k)$ en $p_0$. Par (\ref{eq:31}), on a
$$ E_k = \big( \oplus_{\ell<k} S(\cP_\ell) \big)^\perp . $$
En associant à $\tau\in E_k$ le terme d'ordre $k$ de la solution $s$ de (\ref{eq:32}), on obtient un opérateur linéaire
$$ T:E_k \longrightarrow \cP_k. $$
L'équation (\ref{eq:31}) s'écrit $(\sigma,T\tau) = c (S\sigma,\tau)$ ; comme $S$ est injective, $T$ est surjective, ce qui conclut la démonstration du lemme.
\end{proof}

\subsection*{Le cas des champs de vecteurs}
Le but est de trouver les germes de champs de vecteurs $X$ en $p_0$, satisfaisant
\begin{equation}
  \label{eq:34}
  B\delta^*X = 0,
\end{equation}
et réalisables par un champ de vecteurs global $X$ sur $M_0$, satisfaisant les mêmes équations, et convergeant sur $\partial M_0$ vers un champ de vecteur $X_\infty$ tangent à $\partial M_0$. La réponse est plus compliquée que pour les fonctions, car l'équation (\ref{eq:34}) sur $M_0$ implique que $\delta X=-\tr \delta^*X$ satisfait l'équation infinitésimale d'Einstein sur la partie à trace de la métrique, à savoir $(\frac12\Delta-\Lambda)\delta X=0$, ce qui compte tenu du comportement à l'infini (le poids critique pour $\Delta-2\Lambda$ satisfait $\delta_0^-<-1$ alors que $X=O(x^{-1})$ à l'infini) entraîne la contrainte
\begin{equation}
  \label{eq:35}
  \delta X = 0 .
\end{equation}
On ne peut donc s'attendre à obtenir que des germes satisfaisant (\ref{eq:35}). Notons $\cF_k$ l'espace correspondant de germes, à savoir
\begin{equation}
  \label{eq:36}
  \cF_k = (\Bbb{R}^4\otimes \cP_k)^\Gamma \cap \ker \delta.
\end{equation}

\begin{lemm}\label{lem:champs}
  Tout élément de $\cF_k$ peut être obtenu comme le terme principal en $p_0$ d'un champ de vecteurs $X$ sur $M_0$ satisfaisant $B\delta^*X=0$ et $X\sim \tau$ à l'infini, où $\tau\in C^\infty(\partial M_0,T\partial M_0)$.
\end{lemm}
On prendra garde que $X\sim \tau$ correspond dans nos conventions à un poids $-1$ par rapport à la métrique asymptotiquement hyperbolique $g_0$, puisque $|X|\sim x^{-1}$.
\begin{proof}
  La démonstration est similaire à celle du lemme \ref{lem:fonctions}, mais plus compliquée car l'opérateur $B\delta^*$ a deux paires d'exposants critiques à l'infini, correspondant à la décomposition de $TM_0$ au bord en le fibré tangent à $\partial M_0$ et le fibré normal. Les poids critiques correspondant à $T\partial M_0$ sont $(\delta_1^-,\delta_1^+)=(-1,4)$, alors que ceux correspondant au fibré normal sont $\delta_0^\pm$, avec $\delta_0^+>\delta_1^+$, car les solutions correspondantes sont de la forme $X=df$ avec $(\Delta-2\Lambda)f=0$.

  On commence donc par étudier le problème dual : partant de $\sigma\in \cF_k$, on peut comme dans le lemme \ref{lem:fonctions} fabriquer une solution globale $s$ sur $M_0$ satisfaisant le système équivalent à (\ref{eq:32}) pour l'opérateur $B\delta^*$, à savoir
  \begin{equation}
    \label{eq:37}
    B\delta^*s=0, \qquad s\sim\frac \sigma{r^{2k+2}} \text{ en }p_0, \qquad s \sim x^{\delta_1^+}\tau \text{ à l'infini.}
  \end{equation}
  Posons $S(\sigma)=\tau$, et analysons l'injectivité de $S$. Si $\tau=0$, alors le développement de $s$ va commencer au second poids critique $\delta_0^+$ avec un terme de type $df$ et $f\sim f_\infty x^{\delta_0^+}$. Il est facile de voir que le développement formel de $s$ à l'infini en les puissances de $x$ doit coïncider avec celui de $df$ tel que $(\Delta-2\Lambda)f=0$, ce qui impose que la différentielle extérieure de $s$ vu comme 1-forme s'annule à tout ordre à l'infini, $ds=O(x^\infty)$. Comme $ds$ satisfait aussi $(\Delta-2\Lambda)ds=0$, le théorème de continuation unique \cite{Maz91c} implique $ds=0$ partout. En particulier, près de $p_0$, on a $\frac \sigma{r^{2k+2}}=df$ avec $(\Delta-2\Lambda)f=0$, ce qui impose $f\sim \frac \phi{r^{2k}}$ avec $\phi\in\cP_{k-1}^\Gamma$ et
  $$ \sigma = r^2 d\phi - 2k \phi r dr. $$
  Dans ce cas on calcule $\delta\sigma = 2k\phi$ donc ces solutions sont exactement annulées par la condition de divergence nulle, $\delta\sigma=0$. On considère donc l'opérateur $S(\sigma)=\tau$, défini entre les espaces
  $$ S : \cF_k \longrightarrow \Gamma(T\partial M_0).$$
L'opérateur $S$ est bien défini modulo l'image des opérateurs $S$ sur $\cF_\ell$ pour $\ell<k$. L'injectivité de $S$ implique alors la surjectivité de l'énoncé par la même démonstration que dans le lemme \ref{lem:fonctions}.
\end{proof}

\subsection*{Le cas des 2-tenseurs}
Nous passons à l'équation $Ps=0$ pour $s$ une section de $\Sym_0^2(T^*M_0)$. Étant donné une métrique conforme infinitésimale $\tau$ sur $\partial M_0$, la non dégénérescence de $g_0$ implique qu'on peut résoudre le problème de Dirichlet à l'infini :
\begin{equation}
  \label{eq:38}
  Ps = 0, \qquad s \sim x^{-2}\tau \text{ à l'infini.}
\end{equation}
À nouveau on prendra garde que $x^{-2}\tau$ a une norme qui tend vers une constante à l'infini, donc correspond au poids $\delta_2^-=0$ (et le poids dual est $\delta_2^+=3$). Compte tenu des remarques au début de cette section, les poids critiques de $P$ à l'infini sont exactement les $\delta_0^\pm$, $\delta_1^\pm$ et $\delta_2^\pm$, les deux premiers correspondant à des solutions de type $(\delta^*df)_0$ ou $(\delta^*X)_0$. 

Comme dans le cas des champs de vecteurs, il y a une contrainte sur les solutions de (\ref{eq:38}), provenant de l'annulation du tenseur de Ricci par l'opérateur de Bianchi : on a $BP=-B\delta^*B$ et donc une solution $s$ satisfait $B\delta^*(Bs)=0$ qui implique $Bs=0$. Il est donc naturel de considérer l'espace des germes
\begin{equation}
  \label{eq:39}
  \cG _k = (\Sym_0^2(\Bbb{R}^4)\otimes\cP_k)^\Gamma \cap \ker(B).
\end{equation}
\begin{lemm}\label{lem:extension-sym2}
  Tout élément de $\cG_k$ est obtenu comme le terme principal en $p_0$ d'une solution du problème (\ref{eq:38}).
\end{lemm}
\begin{proof}
  Certains éléments de $\cG_k$ sont déjà connus pour être obtenus comme terme principal en $p_0$ d'une solution de (\ref{eq:38}) : en effet, par le lemme \ref{lem:champs}, les éléments du type $\delta^*X$ sont obtenus globalement comme $s=\delta^*X$ avec $X$ satisfaisant $B\delta^*X=0$ et $X\sim X_\infty$ sur $\partial M_0$, où $X_\infty\in \Gamma(T\partial M_0)$ ; ce qui implique que $s \sim s_\infty$, où $s_\infty$ est l'action infinitésimale de $X$ sur l'infini conforme $[\gamma_0]$. Il est donc naturel de se restreindre à
  \begin{equation}
    \label{eq:40}
    \cG_k^0 = \cG_k \cap (\delta^* \cF_{k+1})^\perp.
  \end{equation}
  Partons de $\sigma\in\cG_k$, alors on peut à nouveau construire à partir de $\frac \sigma{r^{2k+2}}$ un développement formel en $p_0$ pour une solution de $Ps=0$, puis, $P$ étant inversible dans $L^2$ par non dégénérescence de $g_0$, obtenir une solution globale $s$ du système
  $$ Ps = 0, \qquad s \sim \sigma \text{ en }p_0, \qquad s \sim x \tau \text{ à l'infini,} $$
  où $\tau$ est un 2-tenseur symétrique sans trace sur $\partial M_0$ (donc $x\tau$ correspond au poids $\delta_2^+=3$).

  Le but est de définir par $S(\sigma)=\tau$ un opérateur injectif
  \begin{equation}
 S : \cG_k \longrightarrow \Gamma(\Sym_0^2 T^*\partial M_0).\label{eq:41}
\end{equation}
  Il y a une ambiguité sur $\tau$, provenant de la possibilité d'ajouter à $s$ une solution provenant d'un élément de $\cG_\ell$ pour $\ell<k$, donc nous pouvons poser que
  \begin{equation}
   S(\cG_k) \perp \oplus_{\ell<k} S(\cG_\ell).\label{eq:42}
 \end{equation}

 Par la démonstration du lemme \ref{lem:champs}, on sait déjà que $S$ est injective sur $\delta^*\cF_{k+1}$. Prenons à présent $\sigma\in\cG_k^0$, et supposons $S(\sigma)=\tau=0$, alors l'asymptotique de $s$ doit être donnée par le poids suivant $\delta_0^+$, à savoir $\delta_1^+=4$, donc
  $$ s \sim \delta^* X, \qquad X = x^5 X_\infty, X_\infty\in \Gamma(T\partial M_0). $$
(Le passage de la valeur asymptotique de $s$ à $X_\infty$ est algébrique). Or on peut trouver un champ de vecteurs global $Y$ sur $M_0$, solution du problème
$$ B\delta^*Y=0, \qquad Y\sim X_\infty \text{ à l'infini.} $$
Par (\ref{eq:31}) et (\ref{eq:42}), on voit que le terme principal de $Y$ en $p_0$ doit être orthogonal à tous les $\cF_\ell$ pour $\ell\leqslant k$, d'où il résulte que $Y=O(r^{k+1})$ et $\delta^*Y=O(r^k)$ en $p_0$, donc $\delta^*Y \sim \sigma'\in \delta^*\cF_{k+1}$. Appliquant à nouveau (\ref{eq:31}), on obtient
$$ 0=(\sigma,\sigma') = \mathrm{cst.} \|X_\infty\|^2 , $$
d'où résulte $X_\infty=0$. L'asymptotique de $s$ est donc donnée par le poids suivant (et dernier), $\delta_2^+$, ce qui signifie
$$ s \sim (\delta^*df)_0, \qquad f = x^{\delta_2^+}f_\infty, f_\infty\in C^\infty(\partial M_0). $$
La fonction $f$ est une solution asymptotique de $(\Delta-2\Lambda)f=0$, et de $B\delta^*df=0$ on déduit $B(\delta^*df)_0=-\frac14 d\Delta f=-\frac \Lambda2 df$. Plus précisément, $Bs$ possède à l'infini le même développement formel que $df$, avec $f$ solution de $(\Delta-2\Lambda)f=0$ et $f\sim x^{\delta_2^+}f_\infty$. Comme dans la démonstration du lemme \ref{lem:champs}, on déduit que $dBs=O(x^\infty)$ et donc il faut que $dBs=0$ partout. Aussi il existe près de l'infini une fonction $f$ telle que $Bs=df$, $(\Delta-2\Lambda)f=0$ et $s=(\delta^*df)_0$. Si $H^1(M_0,\Bbb{R})=0$, on peut étendre $f$ globalement, mais même si ce n'est pas le cas, on peut l'étendre analytiquement le long de chemins allant jusqu'à $p_0$, et il en résulte que près de $p_0$, il existe une fonction $f$ telle que $s=(\delta^*df)_0$ et $Bs=df$. Mais cela est contradictoire avec le fait que le terme principal $\sigma$ de $s$ en $p_0$ satisfait $B\sigma=0$.

La démonstration de l'injectivité de $S$ s'achève en remarquant que, par le lemme \ref{lem:champs}, chaque élément $\xi\in \cF_{k+1}$ est induit par un champ de vecteurs $X$ sur $M_0$ tel que $B\delta^*X=0$ et $X\sim X_\infty\in \Gamma(\partial M_0)$ ; on peut faire un choix tel que l'application $\xi\mapsto X$ soit linéaire. Notons $F(\xi)\in \Gamma(\Sym_0^2 T^*\partial M_0)$ l'action infinitésimale de $X_\infty$ sur $[\gamma_0]$, alors, par (\ref{eq:31}), on a $(F(\xi),S(\xi'))=\mathrm{cst.} (\delta^*\xi,\delta^*\xi')$ alors que $(F(\xi),S(\sigma))=0$ si $\sigma\in\cG_k^0$. Il en résulte que $S(\cG_k^0)\cap S(\delta^*\cF_{k+1})=0$. La démonstration du lemme se termine alors comme dans le lemme \ref{lem:fonctions}.
\end{proof}

\subsection*{Démonstration de la proposition \ref{prop:germe-A2}}
C'est une application simple du lemme \ref{lem:extension-sym2}. Un germe harmonique homogène d'ordre 4 peut être induit à partir du bord à l'infini, à condition d'être dans le noyau de l'opérateur de Bianchi. Mais tout germe peut être modifié par un difféomorphisme infinitésimal de sorte d'être en jauge de Bianchi. Par conséquent, sans se préoccuper de la condition de jauge, il suffit de trouver un germe harmonique d'ordre 4 en $p_0$ qui modifie non trivialement $\langle(\nabla^2_{11}+\nabla^2_{22}-\nabla^2_{33}-\nabla^2_{44})R(p_0)\omega_1,\omega_1\rangle$. On peut le trouver en utilisant la théorie des représentations de $SO(4)$ : les représentations irréductibles s'écrivent $S_+^kS_-^\ell$ à partir des deux représentations spinorielles fondamentales $S_\pm$, où $S_\pm^k$ désigne le produit symétrique. Ainsi $\Bbb{R}^4=S_+S_-$, $\Omega_\pm=S_\pm^2$, $\Sym_0^2\Bbb{R}^4=S_+^2S_-^2$, etc. L'espace des polynômes harmoniques de degré $k$ est $\cP_k=S_+^kS_-^k$. Alors
$$ \cP_k \otimes \Sym_0^2\Bbb{R}^4 = S_+^k S_-^k \otimes S_+^2 S_-^2 = (S_+^{k+2}\oplus S_+^k\oplus S_+^{k-2})\otimes(S_-^{k+2}\oplus S_-^k\oplus S_-^{k-2}). $$
L'opérateur de Bianchi sera alors à valeurs dans
$$ \cP_{k-1} \otimes \Bbb{R}^4 = S_+^{k-1}S_-^{k-1} \otimes S_+S_- = (S_+^k\oplus S_+^{k-2})\otimes(S_-^k\oplus S_-^{k-2}), $$
tandis que les difféomorphismes infinitésimaux harmoniques sont dans l'espace
$$ \cP_{k+1} \otimes \Bbb{R}^4 = S_+^{k+1}S_-^{k+1} \otimes S_+S_- = (S_+^{k+2}\oplus S_+^k)\otimes(S_-^{k+2}\oplus S_-^k). $$
Au total, si on prend les germes harmoniques dans le noyau de $B$, et orthogonaux aux difféomorphismes infinitésimaux (qui ne modifient pas la courbure), on obtient une description de l'espace noté $\cG_k^0$ plus haut comme
\begin{equation}
  \label{eq:43}
  \cG_k^0 = S_+^{k+2}S_-^{k-2} \oplus S_+^{k-2} S_-^{k+2}.
\end{equation}
Par exemple, pour $k=2$, on obtient $S_+^4\oplus S_-^4$ qui est exactement l'espace des valeurs du tenseur de courbure en $p_0$ si $g_0$ est Einstein, et le lemme \ref{lem:extension-sym2} dit que toute modification de $R(p_0)$ peut être induite par une modification de l'infini conforme, ce qui est le résultat utilisé dans \cite{Biq13}. Plus précisément, $S_\pm^4$ correspond au demi-tenseur de Weyl $W_\pm$.

Pour $k=4$ qui est le but de la proposition \ref{prop:germe-A2}, la flexibilité sur le 2-jet de la courbure est donné par $\cG_2^0=S_+^6S_-^2\oplus S_+^2S_-^6$. La première composante correspond à un espace de dérivées secondes de $W_+$, puisque $W_+$ est une section de $S_+^4$ donc $\nabla^2W_+(p_0) \in (S_+S_-)^2S_+^4 \supset S_+^6S_-^2$. Le terme voulu, $\langle(\nabla^2_{11}+\nabla^2_{22}-\nabla^2_{33}-\nabla^2_{44})R(p_0)\omega_1,\omega_1\rangle$, s'interprète simplement : $e_1^2+e_2^2-e_3^2-e_4^2$ est élément de $S_+^2S_-^2\subset (S_+S_-)^2$, et $\omega_1$ est un élément de $S_+^2$ donc $\omega_1\otimes\omega_1\in (S_+^2)^2\supset S_+^4$, donc un élément non nul de $S_+^6S_-^2$ induisant la courbure voulue est obtenu en projetant l'élément $(e_1^2+e_2^2-e_3^2-e_4^2)\otimes\omega_1^2 \in S_+^2S_-^2S_+^4$ dans $S_+^6S_-^2$. Le résultat est non nul car $e_1^2+e_2^2-e_3^2-e_4^2=(e_1e_2+e_3e_4)\circ(e_1e_2-e_3e_4)=\omega_1\circ(e_1e_2-e_3e_4)$, où $e_1e_2-e_3e_4\in S_-^2$, et la projection de $\omega_1^3$ dans le produit symétrique $S_+^6$ est non nulle. \qed

\section{Autres singularités}
\label{sec:autres-singularites}

\subsection*{Quotient fini de singularités fuchsiennes}
Nous examinons ici le cas de singularités, quotients finis de singularités fuchsiennes, voir la référence \cite{Suv12} ; pour la désingularisation dans le cadre kählérien, voir \cite{BiqRol15}.
Nous ne traitons que le cas le plus simple, à savoir le cas du groupe cyclique $\Bbb{Z}_4$ inclus dans $U(2)$ mais pas dans $SU(2)$. L'espace ALE Ricci plat modèle pour la désingularisation est $T^*\Bbb{R}P^2$, le quotient de l'espace de Eguchi-Hanson $T^*\Bbb{C}P^1$ par une involution.

Supposons donc $(M_0,g_0)$ une variété d'Einstein asymptotiquement hyperbolique, avec une singularité de ce type en $p_0$. Prenons en $p_0$ des coordonnées $(z^1=x^1+ix^2,z^2=x^3+ix^4)$, de sorte que l'action de $\Bbb{Z}_4$ soit engendrée par $\sigma(z^1,z^2)= (-\overline{z^2},\overline{z^1})$. Alors l'action de $\Bbb{Z}_4$ sur $(\Omega_+)_{p_0}$ est non triviale, puisque $\sigma$ agit avec valeurs propres $-1$ sur $\langle\omega_1,\omega_3\rangle$ et $+1$ sur $\omega_2$. Par conséquent, la courbure $R_+^{g_0}(p_0)$ est nécessairement de la forme
\begin{equation}
 R_+^{g_0}(p_0) = \begin{pmatrix} R_{11} & 0 & R_{13} \\ 0 & 0 & 0 \\ R_{31} & 0 & R_{33} \end{pmatrix}.\label{eq:44}
\end{equation}
Le recollement de \cite{Biq13} est réalisé en faisant coïncider à l'infini la structure kählérienne $\omega_1$ correspondant à la structure complexe $T^*\Bbb{C}P^1$ avec un $\omega_1\in (\Omega_+)_{p_0}$ tel que $R_+^{g_0}(p_0)\omega_1=0$. Nous voulons faire ici la même chose, de manière invariante sous l'action de $\Bbb{Z}_2$. Or l'action de $\Bbb{Z}_2$ sur $\omega_1$ dans $T^*\Bbb{C}P^1$ est par $-1$, donc nous avons besoin d'une forme $\omega\in \ker R_+^{g_0}(p_0)\cap \ker(\sigma+1)$. En particulier on ne peut pas choisir $\omega_2$. Les résultats de \cite{Biq13} s'étendent immédiatement :
\begin{prop}\label{prop:Z4}
  Soit $(M_0,g_0)$ une variété d'Einstein asymptotiquement hyperbolique, non dégénérée, avec un point orbifold $p_0$ de type $\frac14(1,1)$, et le groupe local $\Bbb{Z}_4$ est engendré par $\sigma$. Alors on peut désingulariser $(M_0,g_0)$, pourvu qu'existe $\omega\in (\Omega_+)_{p_0}$ telle que
  \begin{enumerate}
  \item $\omega$ est dans l'espace propre de $\sigma$ pour la valeur propre $-1$ ;
  \item $\bR_+^{g_0}(p_0)\omega=0$.
  \end{enumerate}
\end{prop}
Quitte à effectuer une rotation sur $(\omega_1,\omega_3)$, on peut toujours supposer que $\omega=\omega_1$ dans (\ref{eq:44}), et $R_+^{g_0}(p_0)$ doit donc avoir la forme
\begin{equation}
  \label{eq:45}
 R_+^{g_0}(p_0) = \begin{pmatrix} 0 & &  \\ & 0 & \\ & & R_{33} \end{pmatrix}.  
\end{equation}
Les obstructions qui subsistent sur $T^*\Bbb{R}P^2$ sont $o_1$ et $o_3$ ; par conséquent il n'est pas clair a priori que les métriques désingularisées soient non dégénérées, car l'opérateur $\bR_+^{g_0}(p_0)-\Lambda$ de la proposition \ref{prop:L1} admet un noyau sur $\langle\omega_1,\omega_3\rangle$.

Les autres cas sont aussi des quotients finis d'instantons gravitationnels. Il est difficile de les traiter par la méthode de cet article à cause de la dégénérescence de la courbure (\ref{eq:45}), qui empêche l'itération des désingularisations.


\subsection*{Plusieurs points singuliers}
Considérons à présent une variété d'Einstein asymptotiquement hyperbolique $(M_0,g_0)$, non dégénérée, avec plusieurs points orbifolds $p_i$ pour les singularités qu'on sait désingulariser, à savoir $A_1$, $A_2$ et $\Bbb{Z}_4\nsubset SU(2)$ ci-dessus. Supposons que les obstructions à la désingularisation soient nulles. Une étape de la méthode de désingularisation où pourraient interagir les différents points est la production de variations de l'infini conforme $[\gamma_0]$ induisant tous les germes nécessaires aux points singuliers ; dans le cas de plusieurs points, il faut pouvoir induire les germes adéquats au point $p_i$ sans modifier les germes aux autres points $p_j$ pour $j\neq i$. Or les raisonnements de la section \ref{sec:germe-au-point} s'étendent immédiatement au cas de plusieurs points : l'étape cruciale est l'injectivité des opérateurs notés $S$, mais le fait qu'il y ait un ou plusieurs points ne change rien à la démonstration. On en déduit assez rapidement :
\begin{prop}\label{prop:plusieurs-points}
  Supposons $(M_0,g_0)$ asymptotiquement hyperbolique, non dégénérée, avec plusieurs points orbifolds $p_i$ de type $A_1$, $A_2$ et au plus un point de type $\Bbb{Z}_4\nsubset SU(2)$ pour chacun desquels les obstructions à la désingularisation s'annulent. Pour les points de type $A_1$ et $A_2$ on demande en outre qu'ils soient non dégénérés au sens où $\rk \bR_+^{g_0}(p_i)=2$. Alors il existe une désingularisation topologique $M$ et des métriques d'Einstein lisses $g_t$ sur $M$, asymptotiquement hyperboliques, qui désingularisent $g_0$.
\end{prop}
Comme on le verra dans la démonstration, les métriques construites dépendent d'un paramètre par point de type $A_1$ ou $\Bbb{Z}_4$, et de deux paramètres par point de type $A_2$.
\begin{proof}
  Appelons $p_1$, $p_2$,..., $p_k$ les points singuliers, où le point éventuel avec singularité $\Bbb{Z}_4$ est $p_k$. Commençons par désingulariser le point $p_1$ en gardant les singularités aux autres points. Du lemme \ref{lem:extension-sym2} appliqué au cas de plusieurs points résulte que, dans le processus de désingularisation de $p_1$, quitte à perturber l'infini conforme, on peut assurer que les obstructions restent nulles aux autres points $p_i$ pour $i>1$. On obtient donc après désingularisation de $p_1$ une métrique d'Einstein asymptotiquement hyperbolique, non dégénérée, telle que les obstructions continuent à s'annuler aux autres points $p_i$. On peut itérer le processus. On ne met qu'un seul point de type $\Bbb{Z}_4$, et à la fin du processus, car on ne peut plus assurer que la métrique demeure non dégénérée après désingularisation.
\end{proof}

Bien évidemment la proposition précédente n'est pas optimale. Il est plausible qu'on puisse se libérer des hypothèses de non dégénérescence des points $p_i$ et admettre un nombre quelconque de points à singularité $\Bbb{Z}_4$ en modifiant l'analyse faite dans \cite{Biq13}, ce que nous n'avons pas voulu poursuivre ici.


\begin{thebibliography}{10}

\bibitem{And08}
M.~T. Anderson.
\newblock Einstein metrics with prescribed conformal infinity on 4-manifolds.
\newblock {\em Geom. Funct. Anal.}, 18(2):305--366, 2008.

\bibitem{Biq08}
O.~Biquard.
\newblock Continuation unique \`a partir de l'infini conforme pour les
  m\'etriques d'{E}instein.
\newblock {\em Math. Res. Lett.}, 15(6):1091--1099, 2008.

\bibitem{Biq13}
O.~Biquard.
\newblock {D{\'e}singularisation de m{\'e}triques d'Einstein. I}.
\newblock {\em Inventiones Math.}, 192(1):197--252, 2013.

\bibitem{Biq16}
O.~Biquard.
\newblock {D{\'e}singularisation de m{\'e}triques d'Einstein. II}.
\newblock {\em Inventiones Math.}, 204(2):473--504, 2016.

\bibitem{BiqHerRum07}
O.~{Biquard}, M.~{Herzlich}, and M.~{Rumin}.
\newblock {Diabatic limit, eta invariants and Cauchy-Riemann manifolds of
  dimension 3.}
\newblock {\em {Ann. Sci. \'Ec. Norm. Sup\'er. (4)}}, 40(4):589--631, 2007.

\bibitem{BiqRol15}
O.~{Biquard} and Y.~{Rollin}.
\newblock {Smoothing singular constant scalar curvature K\"ahler surfaces and
  minimal Lagrangians.}
\newblock {\em {Adv. Math.}}, 285:980--1024, 2015.

\bibitem{FefGra85}
C.~Fefferman and C.~R. Graham.
\newblock Conformal invariants.
\newblock {\em Ast\'erisque}, (hors s\'erie):95--116, 1985.
\newblock The mathematical heritage of \'Elie Cartan (Lyon, 1984).

\bibitem{Kro89a}
P.~B. Kronheimer.
\newblock {The construction of ALE spaces as hyper-K\"ahler quotients.}
\newblock {\em J. Differential Geom.}, 29(3):665--683, 1989.

\bibitem{Maz91}
R.~Mazzeo.
\newblock Elliptic theory of differential edge operators. {I}.
\newblock {\em Comm. Partial Differential Equations}, 16(10):1615--1664, 1991.

\bibitem{Maz91c}
R.~Mazzeo.
\newblock Unique continuation at infinity and embedded eigenvalues for
  asymptotically hyperbolic manifolds.
\newblock {\em Amer. J. Math.}, 113:25--45, 1991.

\bibitem{MorVia}
P.~Morteza and J.~Viaclovsky.
\newblock {The Calabi metric and desingularization of Einstein orbifolds}.
\newblock \texttt{arXiv:1610.02428}.

\bibitem{Suv12}
I.~{\c Suvaina}.
\newblock {ALE Ricci-flat K\"ahler metrics and deformations of quotient surface
  singularities.}
\newblock {\em {Ann. Global Anal. Geom.}}, 41(1):109--123, 2012.

\end{thebibliography}

\end{document}